\documentclass[10pt]{article}
\usepackage[dvips]{graphicx}
\usepackage{amssymb}
\usepackage{color}
\usepackage{amsmath}
\usepackage{amsthm}
\usepackage{latexsym}
\usepackage[english]{babel}
\usepackage{tikz}
\usetikzlibrary{matrix,arrows,decorations.pathmorphing}

\newcommand{\R}{\mathbb{R}}

\newcommand{\N}{\mathbb{N}}

\newtheorem{theorem}{Theorem}[section]

\newtheorem{corollary}[theorem]{Corollary}
\newtheorem{definition}[theorem]{Definition}

\newtheorem{lemma}[theorem]{Lemma}

\newtheorem{proposition}[theorem]{Proposition}
\newtheorem{remark}[theorem]{Remark}

\begin{document}
\title{Compactness of Dirac-Einstein spin manifolds and horizontal deformations.}
\author{Ali Maalaoui$^{(1)}$ \& Vittorio Martino$^{(2)}$}
\addtocounter{footnote}{1}
\footnotetext{Department of Mathematics, Clark University, 950 Main Street, Worcester, MA 01610, USA. E-mail address:
{\tt{amaalaoui@clarku.edu}}}
\addtocounter{footnote}{1}
\footnotetext{Dipartimento di Matematica, Universit\`a di Bologna. E-mail address:
{\tt{vittorio.martino3@unibo.it}}}

\date{}
\maketitle

\vspace{5mm}

{\noindent\bf Abstract} {\small In this paper we consider the Hilbert-Einstein-Dirac functional, whose critical points are pairs, metrics-spinors, that satisfy a system coupling the Riemannian and the spinorial part. Under some assumptions, on the sign of the scalar curvature and the diameter, we prove a compactness result for this class of pairs, in dimension three and four. This can be seen as the equivalent of the study of compactness of sequences of Einstein manifolds as in \cite{And0,N}. Indeed, we study the compactness of sequences of critical points of the Hilbert-Einstein-Dirac functional which is an extension of the Hilbert-Einstein functional having Einstein manifolds as critical points. Moreover we will study the second variation of the energy, characterizing the horizontal deformations for which the second variation vanishes. Finally we will exhibit some explicit examples. }

\vspace{5mm}


\vspace{5mm}


\vspace{5mm}


\section{Introduction and main results}

\noindent
Let $M$ be a closed (compact, without boundary) manifold of dimension $n$. We recall the energy functional $E$ that couples the gravity with fermionic interactions:
\begin{equation}\label{hilbertdiracenergy}
E(g,\psi)=\int_{M}R_{g}+\langle D_{g}\psi,\psi\rangle-\lambda|\psi|^{2} dv_{g} \; , \qquad \lambda \in \R
\end{equation}
where $g$ is a Riemannian metric on $M$, $\psi$ is a spinor in the spin bundle $\Sigma M$ on $M$, $R_g$ is the scalar curvature, $D_{g}$ is the Dirac operator and $\langle \cdot, \cdot \rangle$ is the compatible Hermitian metric on $\Sigma M$; we will give the precise definitions in the next section.\\
The functional $E$ generalizes the classical Hilbert-Einstein functional and it is invariant under the group of diffeomorphisms of $M$ as well; we address the reader to \cite{Belg, Fin, Kim} for a detailed exposition on this topic. In particular, when restricted to a fixed conformal class, this functional gives rise to equations of Yamabe type, see for instance \cite{guimaamar}, \cite{Maalaoui2013}, \cite{MaaVitt}, \cite{MV2}, \cite{MV3} and the references therein. We point out here the work in \cite{AWW}, where a different energy functional coupling spinors and metric was introduced and studied; its critical points are very restrictive since the spinor component is either a parallel spinor or a Killing spinor (under the constraint). In particular, this shrinks the set of critical points and provides flexibility in the study of their deformations.\\
In this paper we will consider the case with $\lambda >0$, which will imply that the scalar curvature is positive, as we will see. The critical points of this functional are solutions of the system (see next section)
\begin{equation}\label{eq1}
\left\{\begin{array}{ll}
Ric_{g}-\displaystyle\frac{R_{g}}{2}g=T^{g,\psi}\\
\\
D_{g}\psi=\lambda \psi
\end{array}
\right.
\end{equation}
where $T^{g,\psi}$ denotes the energy-momentum tensor
\begin{equation}\label{energy-momentum tensor}
 T^{g,\psi}(X,Y)=-\frac{1}{4}\langle X\cdot\nabla_Y \psi + Y\cdot\nabla_X \psi, \psi \rangle \; ,
\end{equation}
here $\cdot$ and $\nabla$ denote the Clifford multiplication and the metric connection extended to the spin bundle $\Sigma M$.\\
A non-trivial spinor $\psi$  is called an Einstein spinor for the eigenvalue $\lambda$ if it is a solution of the previous system; in addition if $n\geq 3$, by contracting the first equation in (\ref{eq1}), it holds
\begin{equation}\label{Rglambda}
R_{g}=\frac{\lambda}{n-2}|\psi|^{2}_{g} \; .
\end{equation}
We will consider the set
\begin{equation}\label{setofmetrics}
\mathcal{E}(D,c,K)=\{(g,\psi) \in Crit_{c}(E); \quad diam(M,g)\leq D, -\Delta_g R_g \geq -K R_g\} \; ,
\end{equation}
where the constant $c$ introduced in the previous definition is meant to prevent the collapsing of the manifold (see formula (\ref{volumeboundedbelow}) in the sequel). The constant $K>0$ will be useful for bounding the gradient of the spinor $\psi$. Besides their technical use, we recall that in the literature regarding the study of compactness of sets of Riemannian manifolds, several assumptions were considered. For instance, in the classical theory of Gromov-Cheeger, a bound on the curvature tensor is necessary and geometrically meaningful (since without such a bound, conical singularities can form where the curvature blows up), as well as a bound on the diameter an a lower bound on the volume in order to avoid collapsing. These type of assumptions were weakened in several situations, one can check for instance \cite{And0,And1, Ch}. In our setting, the bound $c$, on the energy, is natural from a variational point of view, in the spirit of studying a moduli space of metrics. The diameter bound is again geometrically natural to avoid a scaling non-compactness. The bound $K$ on the Laplacian of the curvature is actually technical.
Notice that if one assumes that the metrics are of constant scalar curvature, then we can drop the dependence on $K$. In fact, one can replace this "analytic" constraint by a geometric one by assuming that the $Q$-curvature is bounded from below.  Indeed, the $Q$-curvature $Q$ has the following formula:
$$Q=-\frac{1}{2(n-1)}\Delta_{g}R_{g}-\frac{2}{(n-2)^{2}}|Ric|^{2}+\frac{n^{3}-4n^{2}+16n-16}{8(n-1)^{2}(n-2)^{2}}R_{g}^{2}.$$
Hence, if $Q$ is bounded from below, then $-\Delta_{g}R_{g}$ is also bounded below in terms of $R^{2}_{g}$ and this is enough for our analysis. In fact, if $(g,\psi)\in Crit_{c}(E)$, then the lower bound on the $Q$-curvature is equivalent to our condition $-\Delta_{g} R_{g}\geq -K R_{g}$.

\noindent
We want to point out that the study of compactness and convergence of manifolds with underlying spinors was investigated in \cite{KL,L}, also some cases of collapsing along the limit were studied in \cite{R1,R2}. In our case, the functional provides more control on the spinorial component and this allow us to keep track of its limit.

\noindent
Our first theorem is a compactness (up to bubbling) type result, in terms of convergence in the Gromov-Hausdorff sense, in the spirit of works \cite{And1,And0,BKN,N}. We have the following:
\begin{theorem}\label{compactnessresults}
Let $n=3$, then the space $\mathcal{E}(D,c,K)$ is compact in the topology induced by the Hausdorff distance. That is, if $(g_{k},\psi_{k})\in \mathcal{E}(D,c,K)$ then there exists a subsequence again denoted by $(g_{k},\psi_{k})$ that converges in $C^{\ell,\alpha}(M)$ to $(g_{\infty},\psi_{\infty})$ for all $\ell>0$ and $0<\alpha<1$ and $\psi_{\infty}$ is an Einstein spinor on $(M,g_{\infty})$.

\noindent
Let $n=4$, then there exist a compact orbifold $(M_{\infty},g_{\infty})$ with a finite set $S$ of orbifold singularities, a spinor $\psi_{\infty}\in \Sigma (M_{\infty}\setminus S)$ and a sequence of $C^{\infty}$ embeddings $F_{k}:(M_{\infty}\setminus S,g_{\infty}) \to (M,g_{k})$, for $k$ large enough, such that
\begin{itemize}
\item $((F_{k})^{*}g_{k},(F_{k})^{*}\psi_{k})$ converges uniformly on compact subsets in the $C^{\ell,\alpha}$ topology on $M_{\infty}\setminus S$, to $(g_{\infty},\psi_{\infty})$ for every $\ell>0$ and $0<\alpha<1$ and $\psi_{\infty}$ is an Einstein spinor on $(M_{\infty}\setminus S,g_{\infty})$.
\item For each $p_{i} \in S$, there exist a sequence of real numbers $r_{k}$ and a sequence of points $x_{k}\in M$ such that $(M,r_{k}g_{k},x_{k})$ converges in the pointed Gromov-Hausdorff sense to $(Y_{i},\overline{g}_{i}, x_{\infty})$, where $(Y_{i},\overline{g}_{i})$ is a Ricci flat non-flat manifold. That is, there exists a sequence of smooth diffeomorphisms $H_{k}: (B(p_{i},r),\overline{g}_{i}) \to (M, r_{k}g_{k})$ so that $(H_{k}^{*}(r_{k}g_{k}))$ converges in the $C^{\ell,\alpha}$ topology in $B(p_{i},r)\subset Y_{i}$, to $\overline{g}_{i}$ for every $r>0$, $\ell>0$ and $0<\alpha<1$.
\item Moreover, there exists a parallel spinor $\psi_{i,\infty}\in \Gamma(\Sigma Y_{i})$ such that $H_{k}^{*}\psi_{k}$ converges in the $C^{\ell,\alpha}$ topology in $B(p_{i},r)$, to the spinor $\psi_{i,\infty}$ for every $\ell>0$ and $0<\alpha<1$ and
$$\liminf_{k\to \infty}\int_{M}|Rm_{g_{k}}|^{2}dv_{g_{k}}\geq \int_{M}|Rm_{g_{\infty}}|^{2}dv_{g_{\infty}}+\sum_{i=1}^{|S|}\int_{Y_{i}}|Rm_{\overline{g}_{i}}|^{2}dv_{\overline{g}_{i}}.$$
\end{itemize}
\end{theorem}


\noindent
Notice that the set $Crit_{c}(E)$ is invariant under the action of the diffeomorphism group and the previous theorem is dedicated to the study of its compactness. Hence, in order to have a better understanding of this set, we would like to know if it has more structure. This is again similar to the study of the moduli space of Einstein manifolds and Einstein deformations as in \cite{And2,And3,Kro}. Hence, we would like to know if transversally to the diffeomorphisms' action, $Crit_{c}(E)$ is a finite dimensional manifold. If this is the case, then combined with the previous theorem, we could have the structure of a compact manifold with boundary. In general the study of moduli spaces is not an easy task because of the action of the diffeomorphisms' group (we refer the reader to \cite{Dai,Wang, AWW}). Our study here is the first step in analyzing the moduli space. Our second result concerns the Dirac-Einstein deformations, which are defined by the vanishing of the second variation of the energy functional at critical points. These are the deformations that does not affect the set of critical points. In particular we characterize the horizontal deformations, namely deformations of the metric at a fixed spinor. This requires actually a carefully procedure since the spin bundle varies with the metric as well: tracking such variations has been done by Bourguignon-Gauduchon in \cite{BG} (see also \cite{AWW}). We will recall it in the sequel; here we first determine the second variation of the functional $E$ at a critical point $(g,\psi)$:

\begin{theorem}\label{secondvariationtransverse} Let $(g+th, \psi +t\varphi)$ denote a path of deformations, with $t\in (-\varepsilon,\varepsilon)$ for a small $\varepsilon>0$, $h$ being transverse to the diffeomorphism action, that is $\delta h=0$, where $(\delta h)_{i}=-(div h)_{i}=-\nabla^{j}h_{ij}$. If $(g,\psi)$ is a critical point of $E$, then we have

\begin{align}
\nabla^{2}E(g,\psi)[(h,\varphi),(h,\varphi)]&=\int_{M}\frac{1}{2}\langle \Delta_{L}h+\nabla\nabla tr(h),h\rangle \notag\\
&+\frac{1}{2}\left( -\Delta tr(h)-\langle Ric_g , h\rangle\right) tr(h) +\frac{R_g}{2}|h|^{2}\notag\\
&+\frac{1}{2}\langle T^{g,\psi},h\rangle tr(h)+\frac{1}{2}\langle h\times T^{g,\psi},h\rangle+\frac{1}{2}\langle\nabla tr(h)\cdot \psi, \varphi \rangle\notag\\
&+\langle \mathcal{D}^{h}\varphi,\psi\rangle+2\langle D_g\varphi-\lambda \varphi,\varphi \rangle dv \notag
\end{align}

\noindent
where $\Delta_{L}$ is the Lichnerowicz Laplacian acting on symmetric 2-tensors, $\nabla \nabla tr(h)$ is the tensor defined by $(\nabla \nabla tr(h))_{ij}=\nabla_{i}\nabla_{j}tr(h)$, $(A\times B)_{ij}=\sum_{k=1}^{n}A_{ik}B_{kj}$ and $\mathcal{D}^{h}\varphi=\sum_{i,j}h_{ij}e_{i}\cdot \nabla_{e_{j}}\varphi$.
\end{theorem}

\noindent
We recall here that if $h$ is a symmetric 2-tensor, then
$$\Delta_{L}h=\nabla^{*}\nabla h+Ric\circ h+h\circ Ric -2\overset{\circ}{R}h,$$
with $Ric\circ h+h\circ Ric$ is defined by considering $Ric$ as a $(1,1)$-tensor, hence
$$(Ric\circ h+h\circ Ric)(X,Y)=h(Ric(X),Y)+h(Ric(Y),X)$$
and
$$(\overset{\circ}{R}h)(X,Y)=\sum_{i,j}R(e_{i},X,Y,e_{j})h(e_{i},e_{j})$$
for an orthonormal basis $(e_{1},\cdots,e_{n}).$

\noindent
As a corollary, one can describe the space of horizontal Einstein-Dirac deformations. Namely, deformations of the metric with fixed spinor that sits in the null space of $\nabla^{2}E(g,\psi)$. Roughly speaking, let us denote by $(g+th, \psi )$ a path of deformations that allows variations only on the metric $g$ at the given spinor $\psi$, then we have

\begin{corollary}\label{horizontaldeformation}
Let the pair $(g,\psi) $ be a critical point for the functional $E$. Then the symmetric 2-tensor $h$ belongs to the space of horizontal Dirac-Einstein deformations, if it solves the following system of equations
$$\left\{\begin{array}{ll}
\delta h=0\\
\\
\left\langle Ric_g-\displaystyle\frac{R_g}{2}g,h\right\rangle=\left\langle T^{g,\psi},h \right\rangle =0\\
\\
\Delta h+R_gh+T^{g,\psi}\times h=0
\end{array}
\right.$$
where in the first equation, $\delta h=0$ means that the deformation of $g$ is transverse to the diffeomorphism action. Moreover, if $\ker\left(\Delta_{g}+\frac{R_g}{2}\right)=0$ or if $g$ is Einstein, then we can replace the second equation by $tr(h)=\langle Ric_g, h\rangle =0$; in particular, the space of horizontal Dirac-Einstein perturbations is finite dimensional.
\end{corollary}

\noindent
In sections 5, we provide some examples of computations of the second variation. For instance, we consider the example of a manifold with a real Killing spinor and the example of Sasakian manifolds with Quasi-Killing spinors.

\noindent
\textbf{Acknowledgment:} The authors want to extend their thanks and gratitude to the referee for the comments and suggestions that led to this improved version of the paper.


\section{Notations and basic formulas}

\noindent
In this section we are going to introduce some definition about spin geometry and we recall some basic formulas that we will need in the sequel. Essentially we will use the same notation as in \cite{Kim}, where we address the reader for the full detailed exposition.\\
Let $\Sigma M$ be the canonical spinor bundle associated to $M$, whose sections are simply called spinors on $M$. This bundle is endowed with a natural Clifford multiplication
$$\text{Cliff}:TM\otimes\Sigma M\longrightarrow \Sigma M, $$
which we will denote by $\cdot$ in the sequel; also we denote the canonical hermitian metric on $\Sigma M$ by $( \cdot, \cdot )$ and the induced metric connection by $\nabla$. We will also denote by $\langle \cdot,\cdot \rangle:=Re (\cdot,\cdot)$, the real part of the hermitian metric which defines an Euclidean dot product on $\Sigma M$. Some useful relations: for $X,Y \in \Gamma(TM), \psi\in \Gamma(\Sigma M)$, where $\Gamma$ is used here to denote smooth sections of the given bundles,
$$\langle X\cdot \psi, Y\cdot \psi \rangle=g(X,Y)|\psi|^2, \quad \langle X\cdot \psi, \psi \rangle=0 ,$$
$$\nabla_X (Y\cdot\psi)=(\nabla_X Y)\cdot\psi + Y\cdot(\nabla_X\psi) \; .$$
If $(e_1, \ldots ,e_n)$ is a local orthonormal frame, it holds, for $k=1, \ldots,n$
$$\nabla_{e_k}\psi = e_k(\psi)-\sum_{i,j=1}^n \Gamma_{kj}^i e_i \cdot e_j \cdot \psi,$$
where $\Gamma_{ki}^j$ are the coefficients of the connection for the frame $(e_1, \ldots ,e_n)$, $\nabla_{e_k} e_i = \sum_{j=1}^n \Gamma_{ki}^j e_j .$\\
We denote by $D :\Sigma M\rightarrow \Sigma M$ the Dirac operator acting on spinors in the following way
$$D \psi=\sum_{i=1}^n  e_i \cdot \nabla_{e_i} \psi \; .$$
Now, let
$$R(X, Y )Z = \nabla_X \nabla_Y Z - \nabla_Y \nabla_X Z -\nabla_{[X,Y ]}Z$$
be the curvature tensor on $M$, in the local orthonormal frame $(e_1, \ldots ,e_n)$, we denote
$$Rm_{ijkl} = Rm(e_i,e_j ,e_k,e_l) = -g(R(e_i,e_j)e_k,e_l),$$
$$ Ric_{jl} = Ric(e_j ,e_l) = \sum_{i=1}^n Rm_{ijil}, \quad R_g=\sum_{i=1}^n Ric_{ii},$$
where $Rm, Ric, R_g$ are the Riemann tensor, the Ricci tensor and the scalar curvature of $M$ respectively. We can extend the curvature tensor on the spinor bundle $\Sigma M$  by
$$R(X, Y ) \psi = \nabla_X \nabla_Y \psi - \nabla_Y \nabla_X \psi -\nabla_{[X,Y ]}\psi ,$$
and it holds
$$ R(X, Y )\psi   = -\frac{1}{2} R(X, Y ) \cdot \psi,$$
where in the right hand side we used the notation for the Clifford multiplication of a 2-form with a spinor, that is:
$$R(X,Y)\cdot \psi =\sum_{i<j} Rm(e_{i},e_{j},X,Y)e_{i}\cdot e_{j}\cdot \psi.$$
Similarly,
$$Ric(X) \cdot \psi= - \sum_{i=1}^n e_i \cdot R(e_i,X)\cdot \psi, \qquad R_g\psi= - \sum_{i=1}^n e_i \cdot Ric(e_i)\cdot \psi .$$
Moreover, if $M$ admits a parallel spinor, that is $\psi \in \Sigma M$ such that $\nabla \psi =0$, then $M$ is Ricci-flat. Let us also write the following formula (Lemma 1.2 in \cite{Kim})
$$2Ric(X) \cdot \psi = D(\nabla_X \psi) - \nabla_X(D\psi)  - \sum_{i=1}^n e_i \cdot \nabla_{(\nabla_{e_i}X)} \psi, \quad X\in \Gamma(TM), \; \psi\in\Gamma(\Sigma M)$$
which implies the Schr\"{o}dinger-Lichnerowicz formula
\begin{equation}\label{Schrodinger-Lichnerowicz}
D^2\psi=-\Delta \psi + \frac{1}{4}R_g\psi ,
\end{equation}
where $-\Delta=\nabla ^{*}\nabla $ is the Laplacian of the connection.  Finally, we recall the formulas for the first variation of the Dirac operator, the volume form and the scalar curvature, in order to compute the variation of the energy functional. The behaviour of the spinor bundle under small changes of the metric has been done by Bourguignon-Gauduchon in \cite{BG}, where they introduced a natural isomorphism to identify spinor bundles related to different metrics on the same manifold. So, we denote by $Sym(0, 2)$ the space of all symmetric (0,2)-tensor fields on $M$ and we will still denote by $\langle \cdot, \cdot \rangle$ the induced metric on $Sym(0, 2)$. It holds
\begin{equation}\label{variationdirac}
\frac{d}{dt}\Big|_{t=0}\langle D_{g+th} \psi_{g+th},\psi_{g+th}\rangle_{g+th}= -\langle T^{g,\psi},h\rangle, \quad h\in Sym(0, 2)
\end{equation}
with $T^{g,\psi}$ the energy-momentum tensor introduced in (\ref{energy-momentum tensor}); here $t$ is a sufficiently small real parameter and we have set $D_{g+th}$ the Dirac operator of the metric $g+th$, also $\psi_{g+th}$ is the push forward of $\psi$ by using the isomorphism defined in \cite{BG}. We recall also, for $ h\in Sym(0, 2)$
\begin{equation}\label{variationvolumecurvature}
\frac{d}{dt}\Big|_{t=0}dv_{g+th}= \frac{1}{2}\langle g,h\rangle dv_g, \quad \frac{d}{dt}\Big|_{t=0} R_{g+th}=-\Delta tr(h)+\delta^{2}(h)- \langle Ric,h\rangle,
\end{equation}
where $\delta^{2}(h)=\sum_{i,j=1}^{n}\nabla^{i}\nabla^{j}h_{ij}$. Finally, putting together formulas (\ref{variationdirac}) and (\ref{variationvolumecurvature}) we obtain the system (\ref{eq1}) for the critical points of the energy functional $E$.


\section{Proof of Theorem \ref{compactnessresults}}

\noindent
Let $\mathcal{E}(D,c,K)$ the set introduced in (\ref{setofmetrics}). For the sake of simplicity, in the sequel sometimes we will omit the dependance from $g$ and $\psi$. We will start with the following
\begin{lemma}
There exists a positive constant $C(n,\lambda)$, depending on the dimension $n$ and the positive real parameter $\lambda$, such that, if $(g,\psi)\in \mathcal{E}(D,c,K)$ then
$$\|\psi\|_{\infty}\leq C(n,\lambda).$$
\end{lemma}
\begin{proof}
By formula (\ref{Rglambda}) and since $D^{2}_{g}\psi =\lambda^{2}\psi$ from the Schr\"{o}dinger-Lichnerowicz equation (\ref{Schrodinger-Lichnerowicz}) we obtain
\begin{equation}\label{formulaspinor}
  -\Delta\psi+\lambda\frac{|\psi|^{2}}{4(n-2)}\psi=\lambda^{2}\psi.
\end{equation}
Now we let $f=\frac{1}{2}|\psi|^{2}$ and we have
$$-\Delta f =\langle -\Delta \psi, \psi \rangle -|\nabla \psi|^{2}\leq -\frac{R}{2}f+2\lambda^{2}f,$$
where the laplacian acting on $f$ is the usual metric laplacian on $M$, the one acting on $\psi$ is the Laplacian of the connection $\nabla ^{*}\nabla $; we denote them in the same way since it is clearly understood which we use. Hence,
$$-\Delta f+\frac{R}{2}f \leq 2\lambda^{2}f.$$
At a maximum point $x_{0}\in M$, we have $\frac{R}{2}\leq 2 \lambda^{2}$. Thus $R\leq 4 \lambda^{2}$. But, again by formula (\ref{Rglambda}) we get
$$|\psi|^{2}\leq 4(n-2)\lambda.$$
\end{proof}

\noindent
Next we prove the following
\begin{proposition}\label{propric}
For any metric $g$ such that $(g,\psi)\in \mathcal{E}(D,c,K)$, we have that $|Ric_{g}|$ is uniformly bounded in terms of $\lambda$ and $n$ and $K$. In particular the Sobolev constant of $g$ is uniformly bounded from below and if $p\in M$ and $r>0$ sufficiently small, then there exists a positive constant $C$ depending on $\lambda, n, K$ and $D$ such that
$$vol(B(p,r))\geq C r^{n}.$$
\end{proposition}

\begin{proof}
Notice that the uniform bound on $\psi$ above, implies a uniform bound on the scalar curvature $R$. The next step is then to find a bound on $|T^{g,\psi}|$. So as above, we compute $-\Delta R$:

\begin{align}
-\Delta R &=-\frac{\lambda}{n-2}\Delta |\psi|^{2}\notag\\
&=\frac{2\lambda}{n-2}\Big(\langle -\Delta \psi, \psi\rangle -|\nabla \psi|^{2}\Big)\notag\\
&=\frac{2\lambda}{n-2}\Big(\left[\lambda^{2}-\frac{R}{4}\right]|\psi|^{2}-|\nabla \psi|^{2}\Big).
\end{align}
Since $(g,\psi)\in \mathcal{E}(D,c,K)$ we have
\begin{align}
|\nabla \psi|^{2}&= \frac{n-2}{2\lambda}\Delta R +\left[\lambda^{2}-\frac{R}{4}\right]|\psi|^{2}\notag\\
&\leq \left[\frac{K}{2}+\lambda^{2}-\frac{R}{4}\right]|\psi|^{2}.\notag
\end{align}
Therefore, $|\nabla \psi|^{2}$ is uniformly bounded. Thus, we see that $|T^{g,\psi}|$ is also uniformly bounded, hence $|Ric|$ is also uniformly bounded in terms of $\lambda$ and $n$ and $K$. Now, the uniform bound on the Ricci curvature and the diameter of the metric $g$ implies, by the result of Croke \cite{Cr}, that the Sobolev constant is uniformly bounded from below. Where here, the Sobolev constant $c_{S}$ is defined as the best constant $c$ satisfying the inequality
$$\|u\|_{L^{\frac{2n}{n-2}}(M)}\leq \frac{1}{c}\|\nabla u\|_{L^{2}(M)}+Vol(M)^{-\frac{2}{n}}\|f\|_{L^{2}(M)}, \forall u \in C^{0,1}(M),$$
if $M$ is compact and
$$\|u\|_{L^{\frac{2n}{n-2}}(M)}\leq \frac{1}{c}\|\nabla u\|_{L^{2}(M)}, \forall u\in C^{0,1}_{c}(M),$$
if $M$ is not compact.  By using the equivalence between the Sobolev inequality and the isoperimetric inequality, as in \cite{N}, we have that
$$vol(\partial B(x,\rho))\geq C_{1}\Big(vol(B(x,\rho))\Big)^{\frac{n-1}{n}},$$
for all $x\in M$ and $\rho>0$ such that $B(x,\rho)\subset B(p,r)$ and where $C_{1}$ is the lower bound on the Sobolev constant. Integrating this last inequality yields the desired result.
\end{proof}

\noindent
We notice that as a byproduct of Proposition \ref{propric}, the total volume of the metric $g$ is uniformly bounded: this follows from the Bishop-Gromov inequality and the uniform bound on the diameter.

\noindent
We focus now on the case $n=4$; in fact, from the Chern-Gauss-Bonnet formula in dimension 4, we have that
$$\chi(M)=\frac{1}{32\pi^{2}}\int_{M}|Rm|^{2}-4|Ric|^{2}+R^{2}\ dv_{g},$$
in particular a uniform bound on the Ricci curvature implies automatically a uniform bound on $\int_{M}|Rm|^{2}dv_{g}$. Therefore we are going to prove a local $L^{\infty}$ bound on $|Rm|$ provided smallness on its $L^{2}$-norm. This will be done through a sequence of propositions.

\begin{proposition}\label{proprim}
Let $(g,\psi)\in \mathcal{E}(D,K,c)$, $p\in M$ and $r>0$ sufficiently small. There exists $\varepsilon_{0}(\lambda,K,D)>0$ such that, if
$$\int_{B_{2r}}|Rm|^{2}dv<\varepsilon_{0},$$
then there exists $C(\lambda,K,D)>0$ such that
\begin{equation}\label{ineq1}
\|Rm\|_{L^{4}(B_r)}\leq C\Big(\|Rm\|_{L^{2}(B_{2r})}+vol(B_{8r})^{\frac{1}{2}}\Big)
\end{equation}
and
\begin{equation}\label{ineq2}
\|\nabla Rm\|_{L^{2}(B_r)}\leq C\Big[\| Rm\|_{L^{2}(B_{2r})}^{2}+\|Rm\|_{L^{2}(B_{2r})}+vol(B_{8r})+vol(B_{8r})^{\frac{1}{2}}\Big],
\end{equation}
where we set $B_r=B(p,r)$.
\end{proposition}

\begin{proof}
Given two tensors $A$ and $B$, we will use $A*B$ to denote a bilinear expression of $A$ and $B$. We recall then from \cite[Lemma~7.4]{H}, that
\begin{equation}\label{laplacianRm}
\Delta Rm=Rm*Rm+\nabla^{2}Ric,
\end{equation}
where
$$(\nabla^{2}Ric)_{ijkl}=-Ric_{jl,ik}+Ric_{il,jk}+Ric_{jk,il}-Ric_{ik,jl} ,$$
where the comma denotes the covariant derivative in local coordinates $(e_1, \ldots ,e_n)$. Since $Ric-\frac{R}{2}g=T^{g,\psi}$, we need to compute the contribution of $T^{g,\psi}$ in $\nabla^{2}Ric$. By the definition (\ref{energy-momentum tensor}) of $T^{g,\psi}$ we have
$$T^{g,\psi}_{ij}=-\frac{1}{4}\Big(\langle e_{i}\cdot \nabla_{e_j}\psi,\psi\rangle+\langle e_{j}\cdot \nabla_{e_i}\psi,\psi\rangle\Big).$$
We set $S_{ij}=\langle e_{i}\cdot \nabla_{e_j}\psi,\psi\rangle$ and we compute the Laplacian of $S_{ij}$. For simplicity, we can assume that we are in inertial coordinates at a given point $x$, that is $\Gamma_{ij}^{k}(x)=0$ and $g_{ij}(x)=\delta_{ij}$.
\begin{align}
\nabla_{e_k}\nabla_{e_k}S_{ij}&=\nabla_{e_k}[\langle e_{i}\cdot \nabla_{e_k}\nabla_{e_j}\psi,\psi\rangle +\langle e_{i}\cdot \nabla_{e_j}\psi,\nabla_{e_k}\psi\rangle+\langle \nabla_{e_{k}}e_{i}\cdot \nabla_{e_{j}}\psi,\psi\rangle]\notag\\
&=\nabla_{e_k}[\langle e_{i}\cdot [\nabla_{e_j}\nabla_{e_k}\psi -\frac{1}{2}R(e_{k},e_{j})(\psi)],\psi\rangle]\notag\\
&+\langle e_{i}\cdot \nabla_{e_k}\nabla_{e_j}\psi,\nabla_{e_k}\psi\rangle +\langle e_{i}\cdot \nabla_{e_j}\psi,\nabla_{e_k}\nabla_{e_k}\psi\rangle\notag\\
&+\langle \nabla_{e_{k}}\nabla_{e_{k}}e_{i}\cdot \nabla_{e_{j}} \psi,\psi\rangle + \langle \nabla_{e_{k}}e_{i}\cdot \nabla_{e_{k}}\nabla_{e_{j}}\psi,\psi\rangle +2\langle \nabla_{e_{k}}e_{i}\cdot \nabla_{e_{j}}\psi,\nabla_{e_{k}}\psi\rangle\notag\\
&=\langle e_{i}\cdot \nabla_{e_k}\nabla_{e_j}\nabla_{e_k}\psi,\psi\rangle +\langle e_{i}\cdot [\nabla_{e_j}\nabla_{e_k}+\nabla_{e_k}\nabla_{e_j}]\psi,\nabla_{e_k}\psi\rangle\notag\\
&+\langle e_{i}\cdot \nabla_{e_j}\psi,\nabla_{e_k}\nabla_{e_k}\psi\rangle-\frac{1}{2}\langle e_{i}\cdot\nabla_{e_k}R(e_{k},e_{j})(\psi),\psi\rangle +Rm*\nabla \psi *\psi\notag\\
&=\langle e_{i}\cdot\nabla_{e_j}\nabla_{e_k}\nabla_{e_k}\psi,\psi\rangle+Rm*\nabla \psi *\psi+\nabla Rm *\psi*\psi\notag \\
&+\langle e_{i}\cdot [\nabla_{e_j}\nabla_{e_k}+\nabla_{e_k}\nabla_{e_j}]\psi,\nabla_{e_k}\psi\rangle+\langle e_{i}\cdot \nabla_{e_j}\psi,\nabla_{e_k}\nabla_{e_k}\psi\rangle.\notag
\end{align}
Summing in $k$, the first term after the last equality yields
$$\langle e_{i}\cdot \nabla_{j}\left[\left(\frac{R}{4}-\lambda^{2}\right)\psi\right],\psi\rangle=\left(\frac{R}{4}-\lambda^{2}\right)\langle e_{i}\cdot\nabla_{j}\psi,\psi\rangle,$$
and the last term yields
$$\left(\frac{R}{4}-\lambda^{2}\right)\langle e_{i}\cdot \nabla_{j}\psi, \psi\rangle.$$
Therefore,
$$\Delta S_{ij}=2\left(\frac{R}{4}-\lambda^{2}\right)S_{ij}+Rm*\nabla\psi *\psi+\nabla Rm*\psi*\psi+\nabla^{2}\psi*\nabla \psi.$$
Notice that since $\Delta \psi \in L^{\infty}$ and since $Ric$ is bounded uniformly, we have from \cite{CZ}, the existence of $C$ depending on $\lambda, K$ and $D$ such that
\begin{align}
\|\nabla^{2}\psi\|_{L^{2}(B_{r})}&\leq C(\|\Delta \psi \|_{L^{2}(B_{2r})}+\|\psi\|_{L^{2}(B_{2r})})\leq CVol(B_{2r})^{\frac{1}{2}}.
\end{align}
Also, recall that $S_{ij}\in L^{\infty}$ uniformly. Hence, if we let $\eta\geq 0$ be a smooth cutoff function such that $\eta=1$ on $B_{r}$ and $\eta=0$ outside $B_{2r}$, there exists $C(\lambda,D)$ such that
\begin{align}
\|\nabla^{2}(\eta S_{ij})\|_{L^{2}(B_{r})}&\leq C\Big(\|\Delta (\eta S_{ij}) \|_{L^{2}(B_{2r})}+\|\eta S_{ij}\|_{L^{2}(B_{2r})}\Big)\notag\\
&\leq C\Big(\|\eta Rm\|_{L^{2}(B_{2r})}+\|\eta \nabla Rm\|_{L^{2}(B_{2r})}+Vol(B_{4r})^{\frac{1}{2}}\Big).
\end{align}
Hence, if we set $H= \nabla^{2}Ric$, we have
$$\|\eta H \|_{L^{2}(B_{r})}\leq C\Big(\|\nabla^{2}(\eta S_{ij})\|_{L^{2}(B_{r})}+Vol(B_{2r})^{\frac{1}{2}}\Big).$$
Therefore,
\begin{equation}\label{est2}
\|\eta H\|_{L^{2}(B_{r})}\leq C(\|\eta Rm\|_{L^{2}(B_{2r})}+\|\eta \nabla Rm\|_{L^{2}(B_{2r})}+Vol(B_{4r})^{\frac{1}{2}}).
\end{equation}
We go back now to the equation (\ref{laplacianRm}) for $Rm$. We have
$$\Delta (\eta Rm)=\eta Rm*Rm+\eta H+\nabla Rm*\nabla \eta .$$
Thus, we get
\begin{align}
\|\nabla (\eta Rm)\|_{L^{2}(B_{2r})}^{2}\leq C\Big(&\|\eta Rm\|_{L^{4}(B_{2r})}\| Rm\|_{L^{2}(B_{2r})}^{2}+\|\eta Rm\|_{L^{2}(B_{2r})}\|\eta H\|_{L^{2}(B_{2r})}\notag\\
&+\|Rm\|_{L^{2}(B_{2r})}\|\eta\nabla Rm\|_{L^{2}(B_{2r})}\Big).\notag
\end{align}
Now using estimate $(\ref{est2})$,
\begin{align}
\|\nabla (\eta Rm)\|_{L^{2}(B_{2r})}^{2}\leq C\Big(&\|Rm\|_{L^{2}(B_{2r})}\|\eta Rm\|_{L^{4}(B_{2r})}^{2}+\|Rm\|_{L^{2}(B_{2r})}\|\nabla (\eta Rm)\|_{L^{2}(B_{2r})}\notag\\
&+\|Rm\|_{L^{2}(B_{2r})}\Big[\|Rm\|_{L^{2}(B_{2r})}+Vol(B_{8r})^{\frac{1}{2}}\Big]\Big),\notag
\end{align}
where we used above that $\nabla (\eta Rm)=\eta \nabla Rm +(\nabla \eta) * Rm$ and $supp(\eta)\subset B_{2r}$.
Using the inequality $2ab\leq \alpha a^{2}+\frac{1}{\alpha}b^{2}$ for a positive $\alpha$ chosen so that $C\alpha <1$, we have the existence of another constant $C(\lambda,K, D)$ such that
\begin{align}\label{est3}
\|\nabla (\eta Rm)\|_{L^{2}(B_{2r})}^{2}\leq & C\Big(\|Rm\|_{L^{2}(B_{2r})}\|\eta Rm\|_{L^{4}(B_{2r})}^{2}\notag\\
&+\|Rm\|_{L^{2}(B_{2r})}\Big[\|Rm\|_{L^{2}(B_{2r})}+Vol(B_{8r})^{\frac{1}{2}}\Big]\Big).
\end{align}
By the Sobolev inequality we have
$$\|\eta Rm\|_{L^{4}(B_{2r})}^{2}\leq C\Big(\|Rm\|_{L^{2}(B_{2r})}\|\eta Rm\|_{L^{4}(B_{2r})}^{2}
+\|Rm\|_{L^{2}(B_{2r})}\Big[\|Rm\|_{L^{2}(B_{2r})}+Vol(B_{8r})^{\frac{1}{2}}\Big]\Big).$$
In particular, there exists $\varepsilon_{0}=\varepsilon_{0}(\lambda,n,D)>0$ such that if $\|Rm\|_{L^{2}(B_{2r})}<\varepsilon_{0}$, then
$$\|Rm\|_{L^{4}(B_{r})}\leq C\Big(\|Rm\|_{L^{2}(B_{2r})}+Vol(B_{8r})^{\frac{1}{2}}\Big).$$
Also, from $(\ref{est3})$, we obtain
$$\|\nabla  Rm\|_{L^{2}(B_{r})}^{2}\leq C\Big(\|Rm\|_{L^{2}(B_{2r})}^{2}+\|Rm\|_{L^{2}(B_{2r})}+Vol(B_{8r})+Vol(B_{8r})^{\frac{1}{2}}\Big).$$
\end{proof}

\noindent
Now we recall the following classical lemma (see \cite[Lemma~4.6]{BKN} for a proof) that we use in the next proposition:
\begin{lemma}\label{classicallemma}
Consider two functions $u\in L^{2}_{loc}$ and $f\in L^{4}_{loc}$ such that $u\geq 0$ and $f\geq 0$ and
$$\Delta u\geq -fu.$$
Then there exists $\varepsilon>0$ and $r>0$ sufficiently small, such that if $\int_{B_{r}}f^{4}<\varepsilon$ then
$$\sup_{B_{\frac{r}{2}}}u\leq C\left(\int_{B_{r}}f^{4}\right)^{\frac{1}{4}}.$$
\end{lemma}

\noindent
We have then
\begin{proposition}\label{propreg}[$\varepsilon$-regularity]
There exist $\varepsilon_{1}(\lambda,K,D)>0$ and $0<r_{0}<1$ such that, if
$$\int_{B_{16r}}|Rm|^{2}dv<\varepsilon_{1}, \quad r<r_{0}, $$
then there exists $C(\lambda,K,D)$ such that
$$\sup_{B_{\frac{r}{2}}}|Rm|\leq C(\lambda,K,D).$$
\end{proposition}

\begin{proof}
The proof is based on the previous Lemma \ref{classicallemma}, considering $u=|Rm|^{2}+1$. Since
$$\Delta u\geq -C(|Rm|+|H|)u,$$
we need an estimate for $\||Rm|+|H|\|_{L^{4}(B_{r})}$. Proposition \ref{proprim} provides us with an estimate for $\|Rm\|_{L^{4}(B_{r})}$, so we will estimate $\|H\|_{L^{4}(B_{r})}$.\\
First, we use the commutation relation between $\Delta$ and $\nabla_{i}$ to write
$$\Delta \nabla_{i}\psi=\nabla_{i}\Delta \psi +Rm*\nabla \psi+\nabla Rm *\psi.$$
Hence, again using \cite{CZ}, we get
$$\|\nabla^{3}\psi\|_{L^{2}(B_{r})}\leq C(\|\Delta \nabla \psi\|_{L^{2}(B_{2r})}+\|\nabla \psi\|_{L^{2}(B_{2r})}),$$
therefore, by Sobolev inequality again, we have
\begin{equation}\label{est2psi}
\|\nabla^{2}\psi\|_{L^{4}(B_{r})}^{2}\leq C\Big[Vol(B_{2r})+\|Rm\|_{L^{2}(B_{2r})}^{2}+\|\nabla Rm\|_{L^{2}(B_{2r})}^{2}\Big],
\end{equation}
where we have used the fact that $-\Delta \psi +\frac{R}{4}\psi=\lambda^{2}\psi$ and the uniform boundedness of $\psi, R$ and $\nabla \psi$. This provides an estimate for the terms of $H$ containing $\nabla^{2}R$. Now we compute $\Delta \nabla_{k} S_{ij}$. So we start by noticing
\begin{align}
\nabla_{k} \Delta S_{ij}= & \nabla R*S_{ij}+\nabla Rm * \nabla \psi * \psi + Rm*\nabla^{2}\psi *\psi \notag\\
&+ Rm*\nabla \psi*\nabla\psi+\nabla^{2}Rm*\psi * \psi+\nabla^{3}\psi*\nabla \psi + \nabla^{2}\psi * \nabla^{2}\psi .\notag
\end{align}
Hence
\begin{align}
\|\nabla_{k} \Delta S_{ij}\|_{L^{2}(B_{r})}\leq & C\Big(Vol(B_{r})^{\frac{1}{2}}+\|\nabla Rm\|_{L^{2}(B_{r})}+\|Rm\|_{L^{4}(B_{r})}\|\nabla^{2}\psi\|_{L^{4}(B_{r})}\notag\\
&+\|\nabla^{2}Rm\|_{L^{2}(B_{r})}+\|\nabla^{3}\psi\|_{L^{2}(B_{r})}+\|\nabla^{2}\psi\|_{L^{4}(B_{r})}^{2}\Big).\notag
\end{align}
But
\begin{align}
\|\nabla^{2}Rm\|_{L^{2}(B_{r})}&\leq C(\|Rm\|_{L^{4}(B_{2r})}^{2}+\|H\|_{L^{2}(B_{2r})})\notag\\
&\leq C(\|Rm\|_{L^{2}(B_{4r})}^{2}+\|\nabla Rm\|_{L^{2}(B_{4r})}+\|Rm\|_{L^{2}(B_{4r})}+Vol(B_{8r})^{\frac{1}{2}}+Vol(B_{16r})).\notag
\end{align}
Therefore
\begin{align}
\|\nabla_{k} \Delta S_{ij}\|_{L^{2}(B_{r})}\leq& C\Big[ \|\nabla Rm\|_{L^{2}(B_{4r})}+\|Rm\|_{L^{2}(B_{4r})}^{2}+\|\nabla Rm\|_{L^{2}(B_{4r})}^{2}\notag\\
&+ \|Rm\|_{L^{4}(B_{r})}\Big(vol(B_{2r})^{\frac{1}{2}}+\|Rm\|_{L^{2}(B_{2r})}+\|\nabla Rm\|_{L^{2}(B_{2r})}\Big)\notag\\
&+\|Rm\|_{L^{2}(B_{4r})}+Vol(B_{16r})^{\frac{1}{2}}\Big]. \notag
\end{align}
Using,
$$\Delta \nabla_{k} S_{ij}=\nabla_{k}\Delta S_{ij}+Rm*\nabla S_{ij}+\nabla Rm * \psi,$$
we get that
\begin{align}
\|\Delta \nabla_{k} S_{ij}\|_{L^{2}(B_{r})}\leq & C\Big[ \|\nabla Rm\|_{L^{2}(B_{4r})}+\|Rm\|_{L^{2}(B_{4r})}^{2}+\|\nabla Rm\|_{L^{2}(B_{4r})}^{2}\notag\\
&+ \|Rm\|_{L^{4}(B_{r})}\Big(vol(B_{2r})^{\frac{1}{2}}+\|Rm\|_{L^{2}(B_{2r})}+\|\nabla Rm\|_{L^{2}(B_{2r})}\Big)\notag\\
&+\|Rm\|_{L^{2}(B_{4r})}+Vol(B_{16r})^{\frac{1}{2}}\Big]. \notag
\end{align}
Therefore, we have
\begin{align}
\|\nabla^{3} S_{ij}\|_{L^{2}(B_{r})}\leq& C\Big[ \|\nabla Rm\|_{L^{2}(B_{8r})}+\|Rm\|_{L^{2}(B_{8r})}^{2}+\|\nabla Rm\|_{L^{2}(B_{8r})}^{2}\notag\\
&+ \|Rm\|_{L^{4}(B_{2r})}\Big(vol(B_{4r})^{\frac{1}{2}}+\|Rm\|_{L^{2}(B_{4r})}+\|\nabla Rm\|_{L^{2}(B_{4r})}\Big)\notag\\
&+\|Rm\|_{L^{2}(B_{8r})}+Vol(B_{32r})^{\frac{1}{2}}\Big]. \notag
\end{align}
Thus, Sobolev inequality yields
\begin{align}\label{est2s}
\|\nabla^{2} S_{ij}\|_{L^{4}(B_{r})}\leq& C\Big[ \|\nabla Rm\|_{L^{2}(B_{8r})}+\|Rm\|_{L^{2}(B_{8r})}^{2}+\|\nabla Rm\|_{L^{2}(B_{8r})}^{2}\notag\\
&+ \|Rm\|_{L^{4}(B_{2r})}\Big(vol(B_{4r})^{\frac{1}{2}}+\|Rm\|_{L^{2}(B_{4r})}+\|\nabla Rm\|_{L^{2}(B_{4r})}\Big)\notag\\
&+\|Rm\|_{L^{2}(B_{8r})}+Vol(B_{32r})^{\frac{1}{2}}\Big],
\end{align}
and combining $(\ref{est2s})$ and $(\ref{est2psi})$ we obtain
\begin{align}
\|H\|_{L^{4}(B_{r})}\leq& C\Big[ \|\nabla Rm\|_{L^{2}(B_{8r})}+\|Rm\|_{L^{2}(B_{8r})}^{2}+\|\nabla Rm\|_{L^{2}(B_{8r})}^{2}\notag\\
&+ \|Rm\|_{L^{4}(B_{2r})}\Big(vol(B_{4r})^{\frac{1}{2}}+\|Rm\|_{L^{2}(B_{4r})}+\|\nabla Rm\|_{L^{2}(B_{4r})}\Big)\notag\\
&+\|Rm\|_{L^{2}(B_{8r})}+Vol(B_{32r})^{\frac{1}{2}}\Big]. \notag
\end{align}
Now using $(\ref{ineq1})$ and $(\ref{ineq2})$ of Proposition \ref{proprim}, assuming that $\|Rm\|_{L^{2}(B_{2r})}<\varepsilon_{0}$, we have
\begin{align}
\|H\|_{L^{4}(B_{r})}\leq& C\Big( Vol(B_{32r})^{\frac{1}{2}}+Vol(B_{32r})+Vol(B_{32r})^{2}+\|Rm\|_{L^{2}(B_{16r})}+\|Rm\|^{2}_{L^{2}(B_{16r})}\notag\\
&+\|Rm\|_{L^{2}(B_{16r})}^{3}+\|Rm\|_{L^{2}(B_{16r})}^{4}\Big).
\end{align}
Hence, there exist $\varepsilon_{1}(\lambda,K,D)$ and $r_{0}(\lambda,K.D)$ such that if $\|Rm\|_{L^{2}(B_{16r})}<\varepsilon_{1}$ and $r<r_{0}$, then
$\||Rm|+|H|\|_{L^{4}(B_{r})}<\varepsilon$. Thus, it follows from Lemma \ref{classicallemma} that
$$\sup_{B_{\frac{r}{2}}}u\leq C(\lambda,K,D).$$
\end{proof}


\noindent
Let us recall now the following definition:
\begin{definition}
Let $(M,g)$ be a Riemannian manifold. We say that $(M,g)$ has an adapted $(r, N,C^{\ell,\alpha})$ (resp. $(r,N,W^{k,p})$) harmonic atlas, if there is a covering $(B(x_{k},r))_{1\leq k\leq N}$ of $M$ made of geodesic balls of radius $r$, for which the balls $B(x_{k},\frac{r}{2})$ cover $M$ and $B(x_{k},\frac{r}{4})$ are disjoint, such that each $B(x_{k},10r)$ has a harmonic coordinate chart $(U_{k},\phi_{k})$, such that the metric tensor in these coordinates is $C^{\ell,\alpha}$ (resp. $W^{k,p}$) and on $B(x_{k},10r)$, there exists a constant $C>1$ such that
$$C^{-1}\delta_{ij}\leq g_{ij}\leq C\delta_{ij}$$
and
$$r^{\ell+\alpha}\|g_{ij}\|_{C^{\ell,\alpha}}\leq C, \quad (\text{ resp. } r^{k-\frac{p}{n}}\|g_{ij}\|_{W^{k,p}}\leq C ).$$
\end{definition}

\noindent
The $C^{\ell,\alpha}$ (resp. $W^{k,p}$)- harmonic radius at $x$, $r_{h}(x)$ is the maximum radius $r$ for which $B_{x}(r)$ has harmonic coordinates.\\
Now we will complete the proof of Theorem \ref{compactnessresults}.

\begin{proof}(of Theorem \ref{compactnessresults}, case $n=4$)
First we notice that if $(g,\psi)\in \mathcal{E}(D,c,K)$ then
\begin{equation}\label{volumeboundedbelow}
c=\int_{M}R_{g}dv_{g}\leq\frac{\lambda}{2}\|\psi\|_{\infty}^{2}vol(M,g)\leq 4\lambda^{2}vol(M,g),
\end{equation}
therefore the volume is uniformly bounded from below. Now given a sequence $(g_{i},\psi_{i})\in \mathcal{E}(D,c,K)$, then, combining the lower bound on the volume with the uniform bound on $Ric_{g}$, $\|Rm\|_{L^{2}}$ and the diameter assumption, we have from \cite[Theorem~2.6]{And1} the convergence of a subsequence of $(M,g_{i})$ in the Gromov-Hausdorff topology to a four dimensional orbifold $(M_{\infty},g_{\infty})$ with finitely many singularities. We can assume without loss of generality that $M_{\infty}\setminus S \hookrightarrow M$. Moreover the convergence to $g_{\infty}$ is in $C^{1,\alpha}$ away from the singularities in the Cheeger-Gromov sense. Also, around the singularity the metric blows up to an ALE (Asymptotically Locally Euclidean) manifold (see \cite[Theorem~1.2]{And1} and the definition therein). We will denote by $S$ the set of singularities and $M^{0}=M_{\infty}\setminus S$ that we identify with $M\setminus S$.

\noindent
Let $r<\frac{r_{0}}{4}$ as defined in Proposition \ref{propreg} and let us consider a covering of $(M,g_{i})$ by balls $B(x_{k},r)$ such that $B(x_{k},\frac{r}{2})$ are disjoint. We let
$$I=\left\{k \in \N; \int_{B(x_{k},16r)}|Rm_{g_{i}}|^{2}dv_{g_{i}}<\varepsilon_{1} \right\},$$
and we define
$$G_{i}(r)=\bigcup_{k\in I}B(x_{k},r).$$
Similarly, we let
$$H_{i}(r)=\bigcup \left\{B(x_{k},r); \int_{B(x_{k},16r)}|Rm_{g_{i}}|_{g_{i}}^{2}\ dv_{g_{i}}\geq \varepsilon_{1}\right\}.$$
Notice that $M_{i}=H_{i}(r)\cup G_{i}(r)$. This is a splitting of $M_{i}$ into a good set where one can control the curvature and a bad set where there is no $L^{\infty}$ control on the curvature. In fact, following \cite{And0,N}, one can show that the number of balls in the definition of $H_{i}(r)$ is finite and uniformly bounded in $i$ and $r<\frac{r_{0}}{4}$.\\
By Propositions \ref{propric} and \ref{propreg}, one has a uniform lower bound on the injectivity radius of $G_{i}$ (see for instance \cite[Section~3]{Ch}). In particular $(G_{i}(r),g_{i})$ has an adapted $(r,N,W^{2,p})$ harmonic atlas with a uniform lower bound on the injectivity radius. Now in a harmonic coordinate patch, we recall that the metric satisfies the quasilinear elliptic equation
$$ -\Delta_{g_{i}}g_{i}+\partial g_{i}*g_{i}=Ric_{g_{i}}.$$
From the first equation in the system (\ref{eq1}), any $g_i$ is a solution of
$$Ric_{g_{i}}=\frac{R_{g_{i}}}{2}\psi_{i}+T^{g_{i},\psi_{i}},$$
and since $\psi_{i}$ satisfies
$$-\Delta_{g_{i}}\psi_{i}+\frac{R_{g_{i}}}{2}\psi_{i}=\lambda^{2}\psi_{i},$$
we have a uniform $C^{1,\alpha}$ bound on $\psi_{i}$ leading to a $C^{0,\alpha}$ bound on $T^{g_{i},\psi_{i}}$. Therefore, we have a uniform $C^{2,\alpha}$ bound on $g_{i}$. This can be naturally iterated to get a $C^{\ell,\alpha}$ bound on $g_{i}$. Therefore, there exists a subsequence of $(G_{i}(r),g_{i})$ that converges in the $C^{\infty}$ topology  to an open manifold $(G_{r},g_{\infty}^{r})$ and $\psi_{i}\to \psi_{\infty}^{r}$ in the $C^{\infty}$ topology.\\
We choose now a sequence $r_{j}\to 0$ such that $r_{j+1}\leq \frac{1}{2}r_{j}$ and we let
$$G_{i}^{j}=\{x\in M; x\in G_{i}(r_{m}) \text{ for some } m\leq j\}.$$
Then we have
$$G_{i}^{1}\subset G_{i}^{2}\subset \cdots \subset (M,g_{i}).$$
We also have that $G_{r_{j}}\subset G_{r_{j+1}}$. So we set $M^{0}=\cup_{j=1}^{\infty}G_{r_{j}}\subset M$. Now using a diagonal argument, we can extract a subsequence of $(G_{i}^{j},g_{i},\psi_{i})$ that converges in $C^{\ell,\alpha}$ to $(M^{0},g_{\infty},\psi_{\infty})$ for all $\ell$. This proves the convergence away from a finite set of points, namely, the singular set. It remains then to study the convergence in the neighborhood of points in this latter set. The singular set $S$ can then be defined by $S=M_{\infty}\setminus M^{0}$. In particular, we can characterize $S$ by
$$S=\left\{\begin{array}{ll}
p\in M, \text{ such that there exists } x_{k}\in M, r>0, \varepsilon_{1}>0 \text{ with} \\
\\
x_{k}\to p \text{ and} \displaystyle \liminf_{r\to 0} \liminf_{k\to \infty}\int_{B(x_{k},r)}|Rm_{g_{k}}|^{2}dv_{g_{k}}\geq \varepsilon_{1}
\end{array}\right\}.$$
Now let us take a point $p_{0}\in S$. We know that there exists  $x_{k}\in (M,g_{k})$ such that $x_{k}\to p_{0}$. Since $S$ is finite, we pick $\delta>0$ so that $B(p_{0},2\delta)\cap S=\{p_{0}\}$ and let
$$r_{k}=\sup_{x\in B(x_{k},\delta)}|Rm_{g_{k}}|,$$
and we notice that $r_{k}\to \infty$. We can then, consider a sequence of pointed Riemannian manifolds $(M,\tilde{g}_{k},x_{k})$ where $\tilde{g}_{k}=r_{k}g_{k}$. Under this new rescaled metric $\tilde{g}_{k}$ we have:
\begin{equation}\label{equalim}
\left\{\begin{array}{ll}
Ric_{\tilde{g}_{k}}-\displaystyle\frac{R_{\tilde{g}_{k}}}{2}\tilde{g}_{k}=\sqrt{r_{k}}T^{\tilde{g}_{k},\psi}\\
\\
D_{\tilde{g}_{k}}\psi=\displaystyle\frac{\lambda}{\sqrt{r_{k}}}\psi.
\end{array}
\right.
\end{equation}
Using the same procedure as above, knowing that now we have $Rm_{\tilde{g}}$ uniformly bounded in every ball $B_{r}\subset B(x_{k},r_{k}\delta)$ we have the convergence of the space $(M,\tilde{g}_{k},x_{k})$ to $(Y,\overline{g},x_{\infty})$ in the pointed Hausdorff topology. In fact, as in \cite{N}, one can show that for every ball $B(x_{\infty},r)\subset Y$, there exists a smooth diffeomorphism $\Phi_{k}:B(x_{\infty},r)\to M$, so that $\Phi_{k}^{*}\tilde{g}_{k}$ converges to $\overline{g}$ in $C^{\ell,\alpha}(B(x_{\infty},r))$ for all $\ell \in \N$. Moreover, since $|Ric_{\tilde{g}_{k}}|^{2}=\frac{1}{r_{k}^{2}}|Ric_{g_{k}}|^{2}$, we see that $\overline{g}$ is indeed a Ricci-flat, non-flat metric and from equation $(\ref{equalim})$, $\psi_{k}$ converges to a harmonic spinor $\psi_{\infty}$ in $C^{\ell,\alpha}(B(x_{\infty},r))$ for every $r>0$. That is on $(Y,\overline{g})$ we have that
$$D_{\overline{g}}\psi_{\infty}=0.$$
Using a result of \cite[Lemma~5.1]{KN}, we know that there $\psi_{\infty}$ is asymptotic to a parallel spinor at infinity. Next, using \cite[Theorem~1.5]{BKN}, we have that the mass of $(Y,\overline{g})$ is zero. Thus, we see that $\psi_{\infty}$ is actually parallel.\\
To prove the last point of Theorem \ref{compactnessresults}, we use again the splitting $M=G_{i}(r)\cup H_{i}(r)$. Thus we have
$$\int_{M}|Rm_{g_{i}}|^{2}_{g_{i}}\ dv_{g_{i}}=\int_{G_{i}(r)}|Rm_{g_{i}}|^{2}_{g_{i}}\ dv_{g_{i}}+\int_{H_{i}(r)}|Rm_{g_{i}}|^{2}_{g_{i}}\ dv_{g_{i}}.$$
First we notice that since the convergence of the metric is in the $C^{k,\alpha}$ sense in $G_{i}(r)$,
$$\liminf_{i\to \infty}\int_{G_{i}(r)}|Rm_{g_{i}}|^{2}_{g_{i}}\ dv_{g_{i}}=\int_{G_{r}}|Rm_{g_{\infty}^{r}}|^{2}_{g_{\infty}^{r}}\ dv_{g_{\infty}^{r}}.$$
On the other hand, since $r$ is arbitrary and by scale invariance,
$$\liminf_{i\to \infty}\int_{H_{i}(r)}|Rm_{g_{i}}|^{2}_{g_{i}}\ dv_{g_{i}}=\liminf_{i\to \infty}\int_{H_{i}(r_{i}r)}|Rm_{\tilde{g}_{i}}|^{2}_{\tilde{g}_{i}}\ dv_{\tilde{g}_{i}}\geq \sum_{k=1}^{|S|}\int_{Y_{k}}|Rm_{\overline{g}_{k}}|^{2}_{\overline{g}_{k}}dv_{\overline{g}_{k}},$$
which leads to the desired inequality.
\end{proof}

\noindent
We notice that in the case $n=3$ we have the compactness result as a corollary from the uniform boundedness of the Ricci curvature proved in Proposition $\ref{propric}$, that is

\begin{corollary}
The space $\mathcal{E}(D,c,K)$ is compact in dimension 3.
\end{corollary}
\begin{proof}
The uniform bound on the Ricci curvature induces a uniform bound on the curvature tensor in dimension 3. Hence, one has automatically a global version of Proposition \ref{propreg}. Now combining this uniform bound with the diameter and the volume bound, provides a uniform lower bound on the injectivity radius.The limiting process is then similar to the one in dimension 4 and the desired result follows (see for instance \cite{And1}, \cite{BKN}, \cite{N}).
\end{proof}

\begin{remark} Let us make the following remarks:

\begin{itemize}

\item[i)] our proof does not require the manifold $M$ to be fixed. Indeed, one can take a sequence $(M_{k},g_{k},\psi_{k})$ and the limiting procedure still works under the same kind of bounds, with an additional upper bound on $b_{2}(M_{k})$, the second Betti number of $M_{k}$. The bound on $b_{2}$ allows to have a uniform upper bound on the $L^{2}$ norm of the curvature tensor.  The rest of the proof follows from Cheeger's finiteness theorem (see \cite{Ch1,Pet}). That is $\mathcal{E}(D,c,K)$ contains only a finite number of diffeomorphism classes;
\item[ii)] we can also take the $\lambda$ to be varying in a compact set $[a,b]$ with $a>0$. The case when $a=0$ is different in nature and was investigated in $\cite{KL}$;
\item[iii)] the constant $c$ in the definition of $\mathcal{E}$ is of a variational nature but it can be replaced by a non-collapsing geometric condition. That is $Vol_{g}M\geq c$;
\item[iv)] the introduction of the constant $K$, was necessary to get a bound on the gradient of the spinor and hence a uniform bound on the Ricci. But notice that one can impose another condition that would replace that bound. For instance if $P^{g}$ is the Penrose operator defined by $P^{g}_{X}\psi=\nabla_{X}\psi+\frac{1}{n}X\cdot D\psi$, one can impose that $|P^{g}\psi|\leq K$. This condition combined with the equation of the spinor (\ref{formulaspinor}) and the uniform bound of $\|\psi\|_{L^{\infty}}$ will lead to a uniform bound on $\|\nabla \psi\|_{L^{\infty}}$. In fact, one can rewrite the first equation of $(\ref{eq1})$ as follows:
$$Ric_g-\frac{R_g}{n}g=\tilde{T}^{g,\psi}_{ij},$$
where
$$\tilde{T}^{g,\psi}_{ij}=-\frac{1}{4}[\langle e_{i}\cdot P^{g}_{e_{j}}\psi,\psi\rangle +\langle e_{j}\cdot P^{g}_{e_{i}}\psi,\psi\rangle];$$
\item[v)] As we mentioned above, the bound $-\Delta_{g} R_{g}\geq -KR_{g}$ can be replaced by a more geometric condition in dimension 4. Namely, the $Q$-curvature $Q_{g}$ is bounded from below by $-K$. Indeed, recall that $Q_{g}=\frac{1}{6}(-\Delta_{g} R+R^{2}-3|Ric_{g}|^{2})$. Therefore, if $Q_{g}\geq -K$, then
$$-\Delta_{g}R \geq -6K-R^{2},$$
and this is enough to carry out the estimates needed in Proposition \ref{propric}.
\end{itemize}
\end{remark}


\section{Second Variation}

\noindent
In this section we want to compute the second variation of $E$ at a critical point $(g,\psi)$. For the sake of simplicity, in the sequel sometimes we will omit the dependance from $g$ and $\psi$ whenever there is no confusion. We first recall the variation of some geometric quantities under the perturbation of the metric. So we consider a variation of metrics $g(t)$ such that $\frac{\partial g(t)}{\partial t}_{|t=0}=h$. Then one has:

$$\left\{\begin{array}{ll}
\displaystyle\frac{\partial Ric_{ij}}{\partial t}_{|t=0}=-\frac{1}{2}\Big(\Delta_{L}h_{ij}+\nabla_{i}\nabla_{j} tr(h)+\nabla_{i}(\delta h)_{j}+\nabla_{j}(\delta h)_{i}\Big)\\
\\
\displaystyle\frac{\partial R}{\partial t}_{|t=0}=-\Delta(tr(h))+\delta^{2}h-\langle Ric, h\rangle,
\end{array}
\right.
$$
where $(\delta h)_{i}=-(div h)_{i}=-\nabla^{j}h_{ij}$, $\delta^{2}(h)=\sum_{i,j=1}^{n}\nabla^{i}\nabla^{j}h_{ij}$ and $\Delta_{L}$ is the Lichnerowicz Laplacia on 2-tensors. Now since we are dealing with variations involving spinors, we need an effective way of differentiation spinors. We state here the results developed in \cite{BGM} (we also refer the reader to \cite[Section~4.2]{AWW}). Consider the manifold $\tilde{M}=I\times M$, where $I$ is a small interval centered at zero, where the variation $g(t)$ is defined. On $\tilde{M}$ we define the metric $\tilde{g}=dt^{2}+g(t)$. We denote by $\nabla^{C}$ the Levi-Civita connection on  $(\tilde{M},\tilde{g})$ and "$\bullet$" the Clifford multiplication on $\Sigma \tilde{M}$. Then we have
\begin{equation}\label{bull}
X\cdot \psi =\nu \bullet X\bullet \psi,
\end{equation}
where $\nu$ is the normal vector to the $M$ as a submanifold of $\tilde{M}$. We can take $\nu=\frac{\partial}{\partial t}$. Therefore, if $X$ and $Y \in TM$, we have
\begin{equation}\label{vecbull}
\nabla^{C}_{X}Y=\nabla_{X}Y+\tilde{g}(W(X),Y)\nu,
\end{equation}
where $W$ is the Weingarten map, i.e. $W(x)=-\nabla_{X}^{C}\nu$. Now on the spinor bundle the covariant derivative can be expressed as follow for $\varphi \in \Sigma \tilde{M}_{|M}$ and $X\in TM$:
\begin{equation}\label{spinbull}
\nabla^{C}_{X}\varphi=\nabla_{X}\varphi-\frac{1}{2}\nu\bullet W(X)\bullet\varphi.
\end{equation}
In our case, as shown in \cite{BGM}, we have
\begin{equation}\label{meth}
\left\{\begin{array}{ll}
\tilde{g}(W(X),Y)=-\frac{1}{2}h(X,Y)\\
\\
Rm_{\tilde{g}}(X,Y,U,\nu)=\frac{1}{2}\Big( (\nabla_{Y}h)(X,U)-(\nabla_{X}h)(Y,U)\Big).
\end{array}
\right.
\end{equation}
Now, we are ready to compute the second variation of $E$.  Let us call
$$E_{1}(g,\psi)=Ric_{g}-\frac{R_{g}}{2}g-T^{g,\psi},$$
and
$$E_{2}(g,\psi)=D_{g}\psi-\lambda \psi.$$
We first compute the variation with respect to $\psi$ since it is easier. So we assume that $g$ is fixed and $\frac{\partial \psi}{\partial t}_{|t=0}=\varphi$. Then we have
$$\frac{\partial E_{1}}{\partial t}_{|t=0}=F(\psi,\varphi) \; ,$$
where the tensor $F$ is defined by
$$F(X,Y)=\frac{1}{4}\Big(\langle X\cdot \nabla_{Y}\varphi +Y\cdot \nabla_{X}\varphi,\psi\rangle+\langle X\cdot \nabla_{Y}\psi +Y\cdot \nabla_{X}\psi,\varphi\rangle\Big).$$
Also,
$$\frac{\partial E_{2}}{\partial t}_{|t=0}=D_{g}\varphi-\lambda\varphi.$$
Now we move to the variation in the $g$ direction as described above. We start by the term $\frac{\partial T_{ij}}{\partial t}$.
So let us consider again the quantity $S_{ij}=\langle e_{i}\cdot \nabla_{e_{j}}\psi,\psi\rangle$ and let us compute its variation. Then using the metric $\tilde{g}$ and assuming that the spinor $\psi$ is parallel transported along the variation and using the fact that $\nabla_{\nu}^{C}\nu=0$ and $(\ref{bull})$, we have

\begin{align}
\frac{\partial S_{ij}}{\partial t}_{|t=0}&=\nabla_{\nu}^{C}\langle e_{i}\cdot \nabla_{e_{j}}\psi,\psi\rangle_{g_{t}}=\langle e_{i}\cdot \nabla_{\nu}^{C}\nabla_{e_{j}}\psi,\psi\rangle_{g_{t}}.\notag
\end{align}
But from $(\ref{spinbull})$ we have
\begin{align}
\nabla_{\nu}^{C}\nabla_{e_{j}}\psi&=\nabla_{\nu}^{C}\Big(\nabla_{e_{j}}^{C}\psi+\frac{1}{2}\nu\bullet W(e_{j})\bullet \psi\Big)\notag\\
&=Rm_{\tilde{g}}(\nu,e_{j})\psi+\nabla^{C}_{[\nu,e_{j}]}\psi+\frac{1}{2}\nu\bullet \nabla^{C}_{\nu}W(e_{j})\bullet \psi\notag\\
&=Rm_{\tilde{g}}(\nu,e_{j})\psi+\nabla_{W(e_{j})}^{C}\psi+\frac{1}{2}\nu\bullet \nabla^{C}_{\nu}W(e_{j})\bullet \psi\\
&=Rm_{\tilde{g}}(\nu,e_{j})\psi+\nabla_{W(e_{j})}\psi+\frac{1}{2}\nu\bullet [\nabla^{C}_{\nu}W(e_{j})-W^{2}(e_{j})]\bullet \psi\notag\\
&=Rm_{\tilde{g}}(\nu,e_{j})\psi+\nabla_{W(e_{j})}\psi+\frac{1}{2}\nu\bullet [Rm_{\tilde{g}}(e_{j},\nu)\nu]\bullet\psi
\end{align}
where here we use the Riccati equation for the Weingarten map $(\nabla_{\nu}^{C}W)(X)=Rm_{\tilde{g}}(X,\nu)\nu+W^{2}(X)$. Now we have (with the convention that $e_{0}=\nu$),
$$Rm_{\tilde{g}}(\nu,e_{j})\psi=-\frac{1}{2}\sum _{0\leq p<\ell \leq n}Rm_{\tilde{g}}(e_{p},e_{\ell},\nu,e_{j})e_{p}\bullet e_{\ell} \bullet \psi$$
and $$Rm_{\tilde{g}}(e_{j},\nu)\nu\bullet \psi=Rm_{\tilde{g}}(e_{j},\nu,\nu,e_{k}) e_{k}\bullet\psi.$$
Hence
$$Rm_{\tilde{g}}(\nu,e_{j})\psi-\frac{1}{2}\nu\bullet [Rm_{\tilde{g}}(e_{j},\nu)\nu]\bullet\psi=-\frac{1}{2}\sum_{1\leq p<\ell\leq n}Rm_{\tilde{g}}(e_{p},e_{\ell},\nu,e_{j})e_{p}\bullet e_{\ell}\bullet \psi.$$
But using $(\ref{meth})$, we have
$$Rm_{\tilde{g}}(e_{p},e_{\ell},\nu,e_{j})=-\frac{1}{2}(\nabla_{\ell}(h)_{pj}-\nabla_{p}(h)_{\ell j}).$$
and $W(e_{j})_{k}=-\frac{1}{2}h_{jk}$. Therefore,
$$\nabla_{\nu}^{C}\nabla_{e_{j}}\psi=\frac{1}{4}\sum_{1\leq p<\ell \leq n}\Big((\nabla_{\ell}h)_{p j}-(\nabla_{p}h)_{\ell j}\Big)e_{p}\cdot e_{\ell}\cdot\psi-\frac{1}{2}h_{j k}\nabla_{e_{k}}\psi.$$
We finally notice that
$$\langle e_{i}\cdot\nabla_{W(e_{j})}\psi,\psi\rangle =-\frac{1}{2}h_{jk}\langle e_{i}\cdot \nabla_{e_{k}}\psi,\psi\rangle =-\frac{1}{2}(h \times S)_{ij},$$
where here we set $(A\times B)_{ij}=\sum_{k=1}^{n}A_{ik}B_{kj}$. Also, rewriting the term below
$$
\frac{1}{4}\langle e_{i}\cdot \sum_{1\leq p<q\leq n}\Big((\nabla_{p}h)_{qj}-(\nabla_{q}h)_{pj}\Big)e_{p}\cdot e_{q} \cdot \psi,\psi\rangle =\frac{1}{4}\sum_{p\not=q\not=i}(\nabla_{p}h)_{qj} \langle e_{i}\cdot e_{p} \cdot e_{q} \cdot \psi ,\psi\rangle,$$
We can define a $3$-tensor $Q$ by $Q(X,Y,Z)=\langle (X\wedge Y \wedge Z)\cdot \psi,\psi\rangle$ in a way that, we can write the later term as $\frac{1}{4} (\nabla h) \times Q$. Therefore, we have
$$\frac{\partial T_{ij}}{\partial t}_{|t=0}=-\frac{1}{2} (h\times T)_{ij}-\frac{1}{16} (\nabla h \times \tilde{Q})_{ij},$$
where $\tilde{Q}$ is the symmetrization of $Q$. But since $Q$ is totally skew-symmetric, we have that $\tilde{Q}=0$. Now, given a symmetric 2-tensor $h$, we define the operator $\mathcal{D}^{h}$ by
$$\mathcal{D}^{h}\varphi=\sum_{i,j=1}^n h_{ij}e_{i}\cdot \nabla_{e_{j}}\varphi.$$
With this notation we have
$$\frac{\partial D_{g}\varphi}{\partial t}_{|t=0}=-\frac{1}{2}\mathcal{D}^{h}\varphi+\frac{1}{4}\nabla(tr(h))\cdot \varphi -\frac{1}{4}\delta(h)\cdot \varphi.$$
To summarize the previous computations, if we denote by $\nabla \nabla tr(h)$ the tensor defined by $(\nabla \nabla tr(h))_{ij}=\nabla_{i}\nabla_{j} tr(h)$ and $sym(\nabla (\delta h))$ the symmetric tensor defined by $(sym(\nabla (\delta h)))_{ij}=\nabla_{i}(\delta h)_{j}+\nabla_{j}(\delta h)_{i}$, we write the second variation of $E$:

\begin{align}
\nabla^{2}E(g,\psi)[(h,\varphi),(h,\varphi)] &=\int_{M}\frac{1}{2}\langle \Delta_{L}h+\nabla \nabla tr(h)+sym(\nabla (\delta h)),h\rangle \notag\\
&+\frac{1}{2}\left( -\Delta tr(h)+\delta^{2}h-\langle Ric , h\rangle\right) tr(h) +\frac{R}{2}|h|^{2}\notag\\
&-\frac{1}{4}\langle \mathcal{D}^{h}\psi,\psi\rangle tr(h)+\frac{1}{2}\langle h\times T,h\rangle\notag\\
&+2\langle -\frac{1}{2}\mathcal{D}^{h}\psi+\frac{1}{4}\nabla tr(h)\cdot \psi-\frac{1}{4}\delta h \cdot \psi, \varphi \rangle\notag\\
&+2\langle F(\varphi,\psi),h\rangle +2\langle D\varphi-\lambda \varphi,\varphi \rangle \ dv.\notag
\end{align}

\noindent
We have
$$\langle F(\varphi,\psi),h\rangle=\frac{1}{2}\langle \mathcal{D}^{h}\psi,\varphi\rangle + \frac{1}{2}\langle \mathcal{D}^{h}\varphi,\psi\rangle$$
and
$$-\frac{1}{2}\langle \mathcal{D}^{h}\psi,\psi\rangle =\langle T,h\rangle.$$
This will allow us to finalize the proof of our result.
\begin{proof}(of Theorem \ref{secondvariationtransverse})
Let us take now a deformation of $g$ that is transverse to the diffeomorphism action, that is $\delta h=0$. Then in the previous formula we have
\begin{align}
\nabla^{2}E(g,\psi)[(h,\varphi),(h,\varphi)] &=\int_{M}\frac{1}{2}\langle \Delta_{L}h+\nabla \nabla tr(h),h\rangle \notag\\
&+\frac{1}{2}\Big( -\Delta tr(h)-\langle Ric_g , h\rangle\Big) tr(h) +\frac{R_g}{2}|h|^{2}\notag\\
&+\frac{1}{2}\langle T^{g,\psi},h\rangle tr(h)+\frac{1}{2}\langle h\times T^{g,\psi},h\rangle+\frac{1}{2}\langle\nabla tr(h)\cdot \psi, \varphi \rangle\notag\\
&+\langle \mathcal{D}^{h}\varphi,\psi\rangle+2\langle D_g\varphi-\lambda \varphi,\varphi \rangle\ dv.\notag
\end{align}
\end{proof}

\noindent
Now we move to the case of horizontal deformations.
\begin{proof}(of Corollary \ref{horizontaldeformation})

\noindent
A horizontal deformation is a deformation of the metric while the spinorial part is fixed. Namely, $\varphi=0$. First recall that
$$R_g=\frac{\lambda}{n-2}|\psi|^{2}.$$
Since the right hand side is invariant under the deformation of the metric, we have that the variation of $R_g$ is zero. Hence
$$-\Delta tr(h)-\langle Ric_g,h\rangle =0.$$
On the other hand, an Einstein-Dirac deformation, is defined by the vanishing of $\nabla^{2}E$. Hence,
$$\Delta_{L}h+\nabla \nabla tr(h)+R_g h+\langle T^{g,\psi},h \rangle g+h\times T^{g,\psi}=0$$
and if we take the trace of this last equation, we get
\begin{equation}\label{traceequation}
2\Delta tr(h)+R_gtr(h)+n\langle T^{g,\psi},h\rangle + \langle T^{g,\psi},h\rangle=0.
\end{equation}
But
$$\langle T^{g,\psi},h\rangle =\langle Ric_g-\frac{R_g}{2}g,h\rangle =-\Delta tr(h)-\frac{R_g}{2}tr(h).$$
Thus, equation (\ref{traceequation}) becomes
$$(1-n)\Delta tr(h)+\frac{R_g}{2}(1-n)tr(h)=0,$$
or equivalently
$$-\Delta tr(h)-\frac{R_g}{2}tr(h)=0.$$
Since $\langle Ric_g,h\rangle =-\Delta tr(h)$, the conclusion follows.
\end{proof}

\begin{remark} Some comments are in order:
\begin{itemize}
\item[i)] We notice that a conformal variation of the metric cannot be an Einstein-Dirac deformation. Indeed, if $h=ug$ then the second condition yields $$\lambda u|\psi|^{2}=0,$$
    hence if $u$ is smooth, then it must be zero.
\item[ii)] The functional $E$ that we are studying has $\lambda$ as a fixed parameter and the Euler-Lagrange equation asserts that $\lambda$ must be an eigenvalue for $D_{g}$ for all critical metrics $g$. In particular if one studies Einstein-Dirac deformations, one needs to take into consideration the stability of the eigenvalue $\lambda$. We recall from \cite[Theorem~24]{BG} that the variation of an eigenvalue of the Dirac operator under the metric reads as
    $$\frac{\partial \lambda}{\partial t}_{|t=0}=-\frac{1}{2}\int_{M}\langle T^{g,\psi},h\rangle\ dv.$$
    Hence, our perturbations $h$ need to satisfy the integral version of $\langle T^{g,\psi},h\rangle =0$.
\item[iii)] Let us go back to the case of general deformation and focus on the ones satisfying $\delta h =tr(h)=0$. This is motivated by the following: if $M$ is not diffeomorphic to the sphere then $T_{g}S^{2}(M)=C^{\infty}(M)g\oplus Im(\delta_{g}^{*})\oplus(\ker(\delta_{g})\cap tr_{g}^{-1}(0))$. The first factor corresponds to conformal changes of the metric and in this setting the problem was investigated in \cite{MV3, guimaamar,BM}. The second factor corresponds to the tangent space of the orbits of the group $Diff_{0}(M)$. So we want to focus on the third factor. Then we have
\begin{align}
\nabla^{2}E(g,\psi)[(h,\varphi),(h,\varphi)] &=\int_{M}\frac{1}{2}\langle \Delta_{L}h,h\rangle +\frac{R_g}{2}|h|^{2}\notag\\
&+\frac{1}{2}\langle h\times T^{g,\psi},h\rangle+ \langle \mathcal{D}^{h}\varphi,\psi\rangle +2\langle D_g\varphi-\lambda \varphi,\varphi \rangle\ dv.\notag
\end{align}

\end{itemize}
\end{remark}

\noindent


\section{Some examples}


\subsection{Example 1: real Killing spinors}

\noindent
We consider a real Killing spinor $\psi$, that is $\nabla_{X}\psi =-\mu X\cdot \psi$ for a real constant $\mu>0$, for all tangent vector $X$. In this case $g$ has to be Einstein and
$$D\psi=-\mu\sum_{i}e_{i}\cdot e_{i}\cdot \psi=n\mu \psi.$$
In particular $\lambda=n\mu$. Also
$$T_{ij}=-\frac{1}{4} [\langle -\mu e_{i}\cdot e_{j}\cdot \psi-\mu e_{j}\cdot e_{i}\cdot \psi,\psi\rangle]=\frac{1}{2}\mu |\psi|^{2}g_{ij},$$
$$R=4n(n-1)\mu^{2}=\frac{n\mu}{n-2}|\psi|^{2}.$$
Thus $|\psi|^{2}=4(n-1)(n-2)\mu$. Now we compute the terms in the second variation:
$$\langle T \times h,h\rangle=\frac{\mu |\psi|^{2}}{2}\langle g_{ik}h_{jk}.h\rangle=\frac{\mu |\psi|^{2}}{2}|h|^{2} =\frac{n-2}{2n}R|h|^{2}=2(n-1)(n-2)\mu^{2}|h|^{2},$$
and
$$\mathcal{D}^{h}\psi=-\mu\sum_{i,j}h_{ij}e_{i}\cdot e_{j}\cdot \psi=\mu tr(h)\psi  . $$
In particular if $tr(h)=0$ we have $\mathcal{D}^{h}\psi=0$. So the second variation reads
\begin{align}
\nabla^{2}E(g,\psi)[(h,\varphi),(h,\varphi)] &=\int_{M}\frac{1}{2}\langle \Delta_{L}h,h\rangle +\frac{R}{2}|h|^{2}+\frac{n-2}{4n}R|h|^{2}+2\langle D\varphi-\lambda \varphi,\varphi \rangle dv\notag\\
&=\int_{M}\frac{1}{2}\langle \Delta_{L}h,h\rangle +\frac{3n-2}{4n}R|h|^{2}+2\langle D\varphi-\lambda \varphi,\varphi \rangle dv .\notag
\end{align}


\subsection{Example 2: Quasi-Killing spinors}

\noindent
Now we consider a Sasakian spin-manifold $(M^{2m+1},g,\phi, \eta)$. A quasi-Killing spinor satisfies
$$\nabla_{X}\psi=aX\cdot \psi+b\eta(X)\xi \cdot \psi,$$
where $\xi$ is the Reeb vector field of $\eta$. We recall from \cite[Theorem 6.6.]{Kim} that if $M$ is simply connected with
$$Ric=\frac{2-m}{m-1}g+\frac{2m^{2}-m-2}{m-1}\eta\otimes \eta,$$
then it does carry a WK-spinor (weak Killing spinor) built from quasi-Killing spinors with
$$a=\pm \frac{1}{2}, \quad b=\mp \frac{2m^{2}-m-2}{4(m-1)}.$$
So let us compute the second variation at such solutions. For the sake of simplicity we will do the computations for $a=\frac{1}{2}$ and $b=-\frac{2m^{2}-m-2}{4(m-1)}$. Notice that in this case we have $R=\frac{2m}{m-1}$. Also, if $\psi$ is an $(a,b)$-quasi-Killing spinor then it is an eigenspinor of the Dirac operator with eigenvalue $\lambda=-(2m+1)a-b$. We calculate now the values of the tensor $T$:
$$T_{ij}=-2g_{ij}+\frac{2m^{2}-m-2}{m-1}\eta\oplus\eta.$$
Also:
$$\mathcal{D}^{h}\psi=\sum_{i,j}h_{ij}e_{i}\cdot \nabla_{j}\psi=b\langle h(\xi)\cdot \xi \cdot \psi,\psi\rangle,$$
where we considered $h$ as an endomorphism, that is $h(\xi)=\sum_{i=1}^{2m+1}h_{i,2m+1}e_{i}$. Now, we compute
$$\langle h \times T,h\rangle =-2|h|^{2}+\frac{2m-m-2}{m-1}|h(\xi)|^{2}.$$
Replacing it in the formula for the second variation of $E$, it yields

\begin{align}
\nabla^{2}E(g,\psi)[(h,\varphi),(h,\varphi)] &=\int_{M}\frac{1}{2}\langle \Delta_{L}h,h\rangle +\frac{m+1}{m-1}|h|^{2}+\frac{2m^{2}-m-2}{2(m-1)}|h(\xi)|^{2}\notag\\
&-\frac{2m^{2}-m-1}{4(m-1)} \langle h(\xi)\cdot\xi \cdot \psi,\varphi\rangle +2\langle D\varphi-\lambda \varphi,\varphi \rangle\ dv.\notag
\end{align}
One might notice the difference here compared to the case of real Killing spinors: if one wants to consider variations that preserves the associate metric to the contact structure $\eta$, then we necessarily have $h(\xi)=0$ and the second variation takes a similar form as the one of a Killing spinor.


\begin{thebibliography}{99}

\bibitem{AWW}B. Ammann, H. Weiss, F. Witt,  A spinorial energy functional: critical points and gradient flow. Math. Ann. 365 (2016), no. 3-4, 1559-1602.

\bibitem{And0} M. Anderson, Ricci curvature bounds and Einstein metrics on compact manifolds. J. Amer. Math. Soc. 2 (1989), no. 3, 455-490

\bibitem{And1} M. Anderson, Convergence and rigidity of manifolds under Ricci curvature bounds. Invent. Math. 102 (1990), no. 2, 429-445.

\bibitem{And2} M. Anderson, Moduli spaces of Einstein metrics on 4-manifolds. Bull. Amer. Math. Soc. (N.S.) 21 (1989), no. 2, 275-279

\bibitem{And3}  M. Anderson, The $L^2$ structure of moduli spaces of Einstein metrics on 4-manifolds. Geom. Funct. Anal. 2 (1992), no. 1, 29-89.

\bibitem{BKN} S. Bando, A. Kasue, H. Nakajima, On a construction of coordinates at infinity on manifolds with fast curvature decay and maximal volume growth. Invent. Math. 97 (1989), no. 2, 313-349.

\bibitem{BGM} C. B\"{a}r, P. Gauduchon, A. Moroianu, Generalized cylinders in semi-Riemannian and Spin geometry. Math. Z. 249 (2005), no. 3, 545–580.

\bibitem{Belg} F.A. Belgun, The Einstein-Dirac equation on Sasakian 3-manifolds, Journal of Geometry and Physics, 37(3), (2001), 229-236.

\bibitem{BM} W. Borrelli, A. Maalaoui, Some Properties of Dirac--Einstein Bubbles, The Journal of Geometric Analysis, (2020), 1-17.

\bibitem{BG}J-P. Bourguignon, P. Gauduchon: Spineurs, opérateurs de Dirac et variations de métriques. Commun. Math. Phys. 144  (1992), 581-599.

\bibitem{Ch1}J. Cheeger, Finiteness theorems for Riemannian manifolds. Am. J. of Math. 92 (1970) 61-74.

\bibitem{Ch} J. Cheeger, Structure theory and convergence in Riemannian geometry. Milan J. Math. 78 (2010), no. 1, 221-264.

\bibitem{Cr} C. Croke, Some isoperimetric inequalities and eigenvalue estimates, Annales scientifiques de l’É.N.S. 4e série, tome 13, no 4 (1980), p. 419-435

\bibitem{Dai} X. Dai, X. Wang, G. Wei, On the stability of Riemannian manifold with parallel spinors. Invent. Math. 161 (2005), no. 1, 151-176

\bibitem{Fin} F. Finster, J. Smoller, S.T. Yau, Particle-like solutions of the Einstein-Dirac equations, {\em Physical Review. D. Particles and Fields.} Third Series 59 (1999).

\bibitem{guimaamar} C.Guidi, A.Maalaoui, V.Martino, Existence results for the conformal Dirac-Einstein system, Advanced Nonlinear Studies (2021), 21, 1, 107-117

\bibitem{CZ} B. Güneysu, S. Pigola, The Calderón-Zygmund inequality and Sobolev spaces on noncompact Riemannian manifolds. Adv. Math. 281 (2015), 353--393

\bibitem{H}R.S. Hamilton, Three-Manifolds with Positive Ricci Curvature. J. Differ. Geom. 17, 255-306 (1982)

\bibitem{KL} V. Kapovitch, J. Lott, On noncollapsed almost Ricci-flat 4-manifolds. Amer. J. Math. 141 (2019), no. 3, 737-755.

\bibitem{Kim} E.C. Kim, T. Friedrich, The Einstein-Dirac Equation on Riemannian Spin Manifolds, {\em Journal of Geometry and Physics}, 33(1-2) (2000), 128-172.

\bibitem{Kro} K. Kr\"{o}ncke, On infinitesimal Einstein deformations. Differential Geom. Appl. 38 (2015), 41-57.

\bibitem{KN} P. Kronheimer, H. Nakajima, Yang-Mills instantons on ALE gravitational instantons. Math. Ann. 288 (1990), no. 2, 263-307.

\bibitem{L} J. Lott, $\hat{A}$-genus and collapsing, J. Geom. Anal. 10 (2000), no. 3, 529–543.

\bibitem{Maalaoui2013} A. Maalaoui, Rabinowitz--Floer homology for superquadratic Dirac equations on compact spin manifolds, {\em Journal of Fixed Point Theory and Applications}, 13(1) (2013): 175-199.

\bibitem{MaaVitt} A. Maalaoui and V. Martino, The Rabinowitz–Floer homology for a class of semilinear problems and applications, Journal of Functional Analysis, 269, (2015), 4006-4037

\bibitem{MV2} A. Maalaoui and V. Martino, Homological approach to problems with jumping non-linearity, {\em Nonlinear Analysis}, 144 (2016) , 165-181.

\bibitem{MV3} A. Maalaoui and V. Martino, Characterization of the Palais-Smale sequences for the conformal Dirac-Einstein problem and applications,	Journal of Differential Equations, 266, 5  (2019), 2493-2541.

\bibitem{N} H. Nakajima, Hausdorff convergence of Einstein 4-manifolds. J. Fac. Sci. Univ. Tokyo Sect. IA Math. 35 (1988), no. 2, 411-424.

\bibitem{Pet} S. Peters, Cheeger’s finiteness theorem for diffeomorphism classes of Riem. mfs. J. reine ang.
Math. 349 (1984) 77-82.

\bibitem{R1} S. Roos, Dirac operators with $W^{1,\infty}$-potential on collapsing sequences losing one dimension in the limit. Manuscripta Math. 157 (2018), no. 3-4, 387–410.

\bibitem{R2} S. Roos, The Dirac operator under collapse to a smooth limit space. Ann. Global Anal. Geom. 57 (2020), no. 1, 121–151

\bibitem{Wang} M. Y. Wang, Preserving parallel spinors under metric deformations. Indiana Univ. Math. J. 40 (1991), no. 3, 815-844.


\end{thebibliography}
\end{document}